\documentclass[preprint]{article}
\pdfoutput=1
\usepackage{graphicx} 
\usepackage{graphicx}
\usepackage{hyperref}
\usepackage{xcolor}
\usepackage{amsthm}
\usepackage{amsmath}
\usepackage{amssymb}
\usepackage[english]{babel}
\usepackage{multicol}
\usepackage{placeins}
\usepackage{tabularx}
\usepackage{geometry}
\title{On the spectral radius properties of a key matrix in periodic impulse control}

\newcommand{\reals}{\mathbb{R}}
\newcommand{\naturals}{\mathbb{N}}

\newcommand{\exponent}{\operatorname{e}}

\newtheorem{stel}{Theorem}
\newtheorem{prop}{Proposition}
\newtheorem{gevolg}{Corollary}
\newtheorem{lemma}{Lemma}

\newtheorem{defi}{Definition}

\newcommand{\deta}{$D\exponent^{\tau A} $}
\newcommand{\rdeta}{r(D\exponent^{\tau A})}

\begin{document}

\maketitle

\begin{center}
Swati Patel\\
Department of Mathematics\\
Oregon State University\\
Corvallis, OR 97331\\
\end{center}

\begin{center}
Patrick De Leenheer\\
Department of Mathematics\\
Oregon State University\\
Corvallis, OR 97331\\
\end{center}

\section*{Abstract}
In this work, we consider the periodic impulse control of a system modeled as a set of linear differential equations.  We define a matrix  that governs the qualitative behavior of the controlled system.  This matrix depends on the period and effects of the control interventions.  We investigate properties of the spectral radius of this matrix and in particular, how it depends on the period of the interventions.  Our main result is on the convexity of the spectral radius with respect to this period.  We discuss implications of this convexity on establishing an optimal and maximum period for effective control. Finally, we provide an example motivated from a real-life scenario.

\section*{Keywords}
\noindent Impulse differential equations, spectral radius, matrix perturbations, biological control, convexity

\section{Introduction}
The main result of this paper is on the convexity of the spectral radius of the matrix \deta, with respect to the scalar $\tau$ under certain conditions on the matrices $D$ and $A$. This matrix arises naturally from considering a linear system with $\tau$-periodic impulse control and its spectral radius allows us to distinguish between distinct qualitative behaviors. Hence, understanding the behavior of the spectral radius with respect to $\tau$ is important for identifying effective and efficient control strategies. 

The convexity of the spectral radius (or eigenvalues) of a given matrix with respect to specific elements or matrix parameters has been a topic of mathematical interest and examined in other contexts. \cite{cohen1978} showed that the spectral radius of a non-negative matrix is convex with respect any single diagonal element and provides applications in demography. These results were extended to show convexity of the spectral radius of non-negative matrices with respect to all diagonal elements by \cite{cohen1979}, \cite{cohen1981}, and \cite{friedland81}, each using distinct approaches which rely on the Feynman-Kac formula, Trotter product forumla \cite{trotter}, and the Donsker-Varadhan variational formula \cite{donsker}, respectively.   Additionally, \cite{friedland81} showed the log convexity of the spectral radius of $\exponent^{D}A$ for non-negative matrices $A$ with respect to diagonal matrices $D$. While these previous convexity results focused on non-negative matrices, our results consider a different subset of matrices, namely diagonally symmetrizable matrices, and we focus on convexity with respect to a different parameter. 

Our problem was motivated from ongoing control interventions of a real biological scenario: massive drug administration (MDA) practices to eliminate parasitic helminth diseases. For this scenario, we have a linear population model, structured into discrete spatial/life classes, in which the control interventions impact the classes differently.  Additionally, we want to know how often to implement control interventions to drive both harmful classes to zero.  

Parasitic helminths are a class of worms that infect and inflict harm on their hosts, including humans \cite{WHO} and several animal species of agricultural importance \cite{Kao}.  These parasites typically spend part of their life cycle in an environment external to the hosts developing before finding their way into a host, where they reproduce and release offspring back into the external environment. Hence, there is exchange between these two spatial locations.
In the 1980s, public health organizations began implementing MDA, the preemptive distribution of anthelmintic drugs to populations where the disease was thought to be prevalent.  The World Health Organization has developed guidelines that recommend MDA periodically 1-2 times a year depending on the prevalence of the disease \cite{WHO_PC}. As drugs are given to hosts, this control intervention kills a subset of the parasitic population within hosts, while the population within the external environment is untouched. 

This example motivated us to examine the spectral radius of the matrix \deta, where $D$ is a diagonal matrix that captures the impact of the impulse control intervention on the distinct classes, $A$ is a matrix that represents the continual transitions between classes coming from the linear model, and $\tau$ is the period of the control interventions. In the following section, we provide more details on the general setting and the extraction of this matrix as well as some assumptions. In section \ref{section_results}, we give the main convexity result.  In section \ref{section_implications}, we give implications of the main result, in particular as they relate to the practical applications. In section \ref{section_examples}, we return to our motivating example and provide insight on their control by directly applying our result. Finally, in section \ref{section_discussion}, we frame these results in a broader context and provide some open questions. 

\section{Setting}\label{section_setting}
\subsection{Model}
Let $x(t)$ be the vector of $N$ state variables at time $t$, which are governed naturally by a linear differential equation, captured in the $N\times N$ matrix $A$.  It is well-known that if $A$ has an eigenvalue with positive real part, then the system is unstable and most solutions will grow unbounded as $t\to +\infty$. 

Periodic impulsive control is a specific control protocol in which control interventions occur at periodically spaced times $n\tau$, where $n$ is a non-negative integer, and where $\tau>0$ is the period of the control. Here we shall consider the scenario where, at the control instants, the current state of the system is modified multiplicatively and instantaneously.
Altogether, we express the dynamics of $x$ as the following impulse differential equation:

\begin{equation}\label{model}
\begin{aligned} 
     \frac{dx}{dt} &= Ax \hspace{2cm} t\not \in n\tau \\
    x(n\tau^+) &= D x(n\tau^-) 
\end{aligned}
\end{equation}
where $x(n\tau^+)$ and $x(n\tau^-)$ denote right and left limits of $x(t)$ as $t$ approaches the control time $n\tau$ from the right and left respectively. Here, $D$ is a fixed $N\times N$ diagonal matrix with positive entries, which often take values in the interval $(0,1]$. 
In this case, the control amounts to re-initializing the state components $x_i$ to a certain fraction $D_{ii}$ of their current value. A state component $x_i$ is uncontrolled if $D_{ii}=1$. Some of the $D_{ii}$ could coincide, but they could also all be distinct.

The main control objective is to stabilize the zero solution of the controlled system. We will soon show that this happens if the spectral radius of the matrix\footnote{The spectral radius of a matrix $A$ is defined as $\max \{|\lambda| \,|\, \lambda \textrm{ is an eigenvalue of }A\}$.}
$
D\exponent^{\tau A}
$, which we denote $\rdeta$,
is less than one. Note that there are two control parameters to achieve this goal: the period $\tau$ between consecutive control interventions, and the diagonal matrix $D$ which encodes how much the state is affected by the control. Thus, we are interested in 
understanding the behavior of the map $(\tau,D)\to r(D\exponent^{\tau A})$, and in particular whether the range of this map includes values below $1$. A secondary control objective might be to minimize 
$r(D\exponent^{\tau A})$, especially when the aim is to stabilize the system as efficiently as possible, i.e., optimize the control.

In this paper, we focus on the restricted scenario in which the matrix $D$ is fixed, and only the period $\tau$ is variable. In other words, we shall consider the map $\tau \to r(D\exponent^{\tau A})$, assuming that $D$ (and of course $A$) is fixed. In many practical applications, $D$ is indeed fixed. For example, some of the diagonal entries of $D$ might reflect efficiencies of existing drugs used to treat a disease, or the potencies of insecticides used to protect an agricultural crop. In such cases, the only control parameter available for tuning is the period $\tau$. For scenarios where $\tau$ is fixed, but $D$ is tunable, we refer to the theory developed in \cite{friedland81,cohen1981} which shows how 
$r(D\exponent^{\tau A})$ varies qualitatively with respect to $D$.

To conclude this section, we discuss how the periodic impulsive control formalism can be analyzed by means of Floquet's theory \cite{chicone} and lead us to consider $\rdeta$. 
We will show that the behavior of the impulsively controlled model (\ref{model}) can also be described by that of the periodically time-varying system
$$
{\dot y}=B(t)y,
$$
where
$$
B(t)=\begin{cases} 
\ln(D),\textrm{ if } t\in [n\tau, n\tau+1)\\
A,\textrm{ if } t\in [n\tau +1,(n+1)\tau + 1)
\end{cases}
$$
for all non-negative integers $n$, and 
$\ln(D)$ is a diagonal matrix defined by $[\ln D]_{ii}=\ln(D_{ii})$ for all $i=1,\dots,N$. Note that $B(t)$ is periodic with period $\tau+1$, and piece-wise constant with time $t$.

We have replaced the instantaneous control map which maps the state $x(n\tau -)$ of the impulsively controlled system to the state $x(n\tau +)=Dx(n \tau -)$, by a continuous-time system ${\dot y}=\ln(D)y$ which acts on the time intervals $[n\tau, n\tau +1)$. During these intervals, the state $y(n\tau )$ evolves to the state   $y(n\tau +1)=\exponent^{\ln(D)}y(n\tau )=Dy(n\tau)$. Therefore, the image of the state $y(n\tau)$ to $y(n\tau+1)$ for the above periodically time-varying system is the same as the map of $x(n\tau -)$ to $x(n\tau+)=Dx(n\tau -)$ for the impulsively controlled system. Also, note that if $x(t)$ is a solution of the impulsively controlled system, defined on some interval $[n\tau,(n+1)\tau]$, then $x(t)$ coincides with a solution $y(t)$ of the system ${\dot y}=B(t)y$, but defined on the interval $[n\tau +1,(n+1)\tau +1]$, whenever the initial condition of both functions coincides. Indeed, on these intervals, both functions are solutions of the same linear system ${\dot z}=Az$, which is known to have unique solutions for every initial condition.

The advantage is that we now can appeal to Floquet's theory to analyze the periodic linear system ${\dot y}=B(t)y$. It has the following principal fundamental matrix solution 
$$
\phi(t)=\begin{cases}\exponent^{t\ln(D)},\textrm{ if }t \in [0, 1) \\
\exponent^{(t-1)A}D,\textrm{ if } t\in [1,1+\tau]
\end{cases},
$$
and the monodromy matrix associated to $\phi(t)$ is
$$
\phi^{-1}(0)\phi(\tau+1)=\exponent^{\tau A}D.
$$

The eigenvalues of this monodromy matrix are the characteristic multipliers. If all characteristic multipliers have modulus less than $1$, or equivalently\footnote{For any pair of $N\times N$ matrices $A$ and $B$, the eigenvalues, and hence the spectral radius, of $AB$ and $BA$ coincide.}, if $r(\exponent^{\tau A}D)=r(D\exponent^{\tau A})<1$, then the zero solution of ${\dot y}=B(t)y$ is asymptotically stable. In this case, the zero solution of the original impulsively controlled system (\ref{model}) is asymptotically stable as well. Hence, we are motivated to understand the properties of $\rdeta$.

\subsection{Definitions and Preliminaries}
To set up our main result, we provide some definitions and useful preliminaries. We begin with a brief discussion of diagonally symmetrizable matrices and then give three different notions of convexity, which we use in our main results.

\begin{defi}
A real $N\times N$ matrix $A$ is diagonally symmetrizable \footnote{Our definition and terminology differ slightly from those in \cite{mckee}, but it is not hard to show that both definitions are equivalent.} if there exists a real, invertible diagonal matrix $T$ such that $T^{-1}AT$ is a symmetric matrix.
\end{defi}

As a diagonally symmetrizable matrix is similar to a symmetric matrix, all its eigenvalues must be real. This is a strong restriction on such matrices. On the other hand, we will see that they can be characterized in an elegant way, enabling us to recognize them fairly easily.

For any real number $x$ we define the sign function
$$
\textrm{sgn}(x)=\begin{cases}
+1,\textrm{ if }x>0\\
0,\textrm{ if }x=0\\
-1,\textrm{ if }x<0
\end{cases}
$$

\begin{defi}
A real $N\times N$ matrix $A$ is sign-symmetric if 
$\textrm{sgn}(A_{ij})=\textrm{sgn}(A_{ji})$, for all $i\neq j$.
\end{defi}

It follows from the definition that if $A$ is diagonally symmetrizable, then there is an invertible diagonal matrix $T$ such that
\begin{equation}\label{balance}
A_{ij}=A_{ji}\left(\frac{T_{ii}}{T_{jj}}\right)^2,\textrm{ for all }i\neq j.
\end{equation}
This implies that $\textrm{sgn}(A_{ij})=\textrm{sgn}(A_{ji})$, for all $i\neq j$, and therefore $A$ is sign-symmetric. Furthermore, for any distinct elements $i_1,\dots,i_k$ in $\{1,\dots,N\}$,  
it follows from $(\ref{balance})$ that
\begin{equation}\label{cycle}
A_{i_1i_2}A_{i_2i_3}\dots A_{i_{k-1}i_k}A_{i_k i_1}=
A_{i_1 i_k}A_{i_k i_{k-1}}\dots A_{i_3 i_2}A_{i_2 i_1},
\end{equation}
known as the cycle condition \cite{mckee}. To justify this terminology we offer a simple graphical interpretation: Associating a directed and weighted graph to $A$ by drawing an edge from vertex $i$ to vertex $j$ with weight $A_{ij}$, the cycle condition $(\ref{cycle})$ expresses that in any directed cycle $i_1\to i_2 \to \dots \to i_{k-1} \to i_k \to i_1$ in this graph, the weight of the cycle (defined as the product of the weights of all the directed edges in the cycle) equals the weight of the reversed cycle $i_1\to i_k \to i_{k-1}\to \dots i_2 \to i_1$.

Remarkably, it turns out that if a sign-symmetric matrix satisfies $(\ref{cycle})$, then it is diagonally symmetrizable \cite{mckee}. In summary, the following elegant characterization of diagonally symmetrizable matrices is available.
\begin{prop} (Proposition 5.15 in \cite{mckee}) \label{diag-symm-char}
A real $N\times N$ matrix $A$ is diagonally symmetrizable if and only if 
$A$ is sign-symmetric and satisfies the cycle condition $(\ref{cycle})$.
\end{prop}

\begin{defi}
An $N\times N$ matrix $A$ is tri-diagonal if $A_{ij}=0$ for all $|i-j|>1$.
\end{defi}
Proposition $\ref{diag-symm-char}$ immediately implies
\begin{gevolg}\label{diag-symm-tri}
Every sign-symmetric, tri-diagonal matrix is diagonally symmetrizable.
\end{gevolg}

The following provides a sense of how restrictive it is for a matrix to be diagonally symmetrizable. It turns out that it is not very restrictive for matrices of size two, but is restrictive for matrices of size larger than two, in a sense made precise below.
\begin{itemize}
\item
We say that a real $N\times N$ matrix is strictly cooperative\footnote{Cooperative matrices are also known as Metzler matrices, especially in  control theory.} (respectively strictly competitive) if $A_{ij}>0$ (respectively $A_{ij}<0$) for all $i\neq j$. 
Note that being a strictly cooperative (or competitive) matrix is an open condition: Any strictly cooperative (or competitive) matrix has a neighborhood of strictly cooperative (competitive)  matrices, under a standard topology. 

Strictly cooperative and competitive matrices frequently occur in the analysis of biological and chemical systems. 

Corollary $\ref{diag-symm-tri}$ implies that all strictly cooperative and all strictly competitive $2\times 2$ matrices are diagonally symmetrizable.
\item Assume that $A$ is a $3\times 3$ matrix which is sign-symmetric. 
For instance, $A$ may be strictly cooperative or competitive, although it doesn't have to be of either type; for example, $A_{12}$ and $A_{21}$, and $A_{13}$ and $A_{31}$ could all be positive, while $A_{23}$ and $A_{32}$ could be negative.

Proposition $\ref{diag-symm-char}$ then implies that $A$ is diagonally symmetrizable if and only if 
$$
A_{12}A_{23}A_{31}=A_{13}A_{32}A_{21}.
$$
It follows that being a diagonally symmetrizable $3\times 3$ matrix is non-generic. Indeed, every neighborhood of such a matrix contains matrices which violate the above cycle condition.
\end{itemize}
For more on properties and the structure of diagonally symmetrizable matrices, see \cite{mckee}.

To state our main convexity results, we define three commonly used notions of convexity; further relevant properties of these notions are reviewed in the Appendix.

\begin{defi}
    Let $C$ be a convex set in $\mathbb{R}^N$.  A function $f: C \rightarrow \mathbb{R}$ is
    \begin{itemize}
        \item \textbf{convex} if for all $x_1, x_2 \in C$ and $\lambda \in [0,1]$
        $$
        f(\lambda x_1 +(1-\lambda)x_2) \leq \lambda f(x_1) +(1-\lambda)f(x_2)
        $$
        \item \textbf{strictly convex} if for all $x_1\neq x_2 \in C$ and $\lambda \in (0,1)$
        $$
        f(\lambda x_1 +(1-\lambda)x_2) < \lambda f(x_1) +(1-\lambda)f(x_2)
        $$
        \item \textbf{strongly convex} with parameter $m>0$ if $f(x) - \frac{m}{2}||x||^2$ is convex.
    \end{itemize}
\end{defi}

Here, $||x||$ denotes the Euclidean norm of the vector $x$. 
These characterizations are ordered from weakest to strongest. It is easy to see that strict convexity implies convexity and that strong convexity implies strict convexity. None of the converse implications hold. For example, if $f:\reals \to  \reals$ is defined as $f(x)=x$, then $f$ is convex, but not strictly convex. If $f: (-1,1)\to \reals$ is defined as $f(x)=x^4$, then $f$ is strictly convex, but we show in the Appendix that $f$ is not strongly convex for any positive parameter $m$.

\section{Main Results}\label{section_results}

We are interested in the properties, including the convexity, of the map $\tau \rightarrow r(De^{\tau A})$, because of its implications for identifying how frequently one needs to intervene to effectively and efficiently control the state variables. 



The main technical result of this paper is as follows.
\begin{stel} \label{main}
Assume that
\begin{enumerate}
    \item $D$ is an $N\times N$ diagonal matrix with positive diagonal entries, and 
    \item $A$ is a diagonally symmetrizable $N\times N$ matrix.
\end{enumerate}
Then the map $\tau \to r(D\exponent^{\tau A})$, is convex for $\tau \in [0,+\infty)$.

If in addition, $A$ is non-singular, then for every $p>0$, there exists some 
$m_p>0$ such that the map $\tau \to r(D\exponent^{\tau A})$ is strongly convex with parameter $m_p$ for $\tau\in [0,p]$.
\end{stel}

This result is easily shown when $D_{ii}=D_{jj}$ for all $i= 1, ...,N$ and $j = 1, ...,N$.  In this case,

$$
\rdeta = D_{11}e^{\tau \lambda_{\max}(A)}
$$
where for any diagonally symmetrizable matrix $S, \lambda_{\max}(S)$ denotes the largest eigenvalue of $S$. (Recall that since $S$ is diagonally symmetrizable, all eigenvalues of $S$ are real.)
Moreover, if $A$ is non-singular, then $\lambda_{\max}(A)\neq 0$ and this map is either a strictly increasing or strictly decreasing exponential, which is strongly convex on 
every compact interval $[0,p]$.
The more interesting case is when the diagonal elements of $D$ are not all equal (from a practical stand point, when the control intervention impacts different classes differently). The proof is as follows.

\begin{proof}
\begin{itemize}
\item First we will show that $\tau\to r(D\exponent^{\tau A})$ is convex for 
$\tau \in [0,+\infty)$.

Since $A$ is diagonally symmetrizable, there is a real diagonal and invertible matrix $T$ such that the matrix ${\tilde A}:=T^{-1}AT$, is symmetric. Then $D\exponent^{\tau A}$ 
is similar to 
$$
T^{-1}D\exponent^{\tau A}T=D\exponent^{\tau {\tilde A}},
$$
where the equality holds because $T^{-1}$ and $D$ commute as both are  diagonal. 
Let $D^{1/2}$ be the unique, diagonal square root of $D$. Then $D\exponent^{\tau {\tilde A}}$, and hence also $D\exponent^{\tau A}$, is similar to
$$
D^{1/2}\exponent^{\tau {\tilde A}}D^{1/2},
$$
which is a symmetric matrix. 
Furthermore, we claim that $D^{1/2}\exponent^{\tau {\tilde A}}D^{1/2}$ is positive definite. To see why, let $\lambda_i$, $i=\dots,N$, 
be the eigenvalues of ${\tilde A}$ (and also of $A$). Note that since ${\tilde A}$ is symmetric, all its eigenvalues are real and additionally, $\exponent^{\tau {\tilde A}}$ is also symmetric. By the Spectral Mapping Theorem, $\exponent^{\tau\lambda_i}$, $i=1,\dots,N$, are 
the eigenvalues of $\exponent^{\tau {\tilde A}}$ and they clearly are all positive for all $\tau \in [0,\infty)$. Thus, $\exponent^{\tau {\tilde A}}$ is positive definite for all $\tau \in [0,\infty)$. But then 
$D^{1/2}\exponent^{\tau {\tilde A}}D^{1/2}$ is also positive definite, as claimed. Consequently,
$$
r(D^{1/2}\exponent^{\tau {\tilde A}}D^{1/2})=\lambda_{\max}(D^{1/2}\exponent^{\tau {\tilde A}}D^{1/2}),
$$
Since $D\exponent^{\tau A}$ is similar to 
$D^{1/2}\exponent^{\tau {\tilde A}}D^{1/2}$, it follows that
$$
r(D\exponent^{\tau A})=\lambda_{\max}(D^{1/2}\exponent^{\tau {\tilde A}}D^{1/2}).
$$

By the Raleigh quotient formula,
$$
r(D\exponent^{\tau A})=\max_{x: \langle x,x \rangle =1} \langle x, D^{1/2}\exponent^{\tau {\tilde A}}D^{1/2} x\rangle,
$$
where $\langle x,y\rangle$ denotes the Euclidean inner product of any two vectors $x$ and $y$ in $\reals^N$.

Since ${\tilde A}$ is symmetric, the Spectral Theorem implies that there exists an orthogonal matrix $Q$ (i.e., $QQ^T=Q^TQ=I$) and a diagonal matrix $\Lambda$ such that
$$
Q^T {\tilde A}Q= \Lambda,\textrm{ and }\Lambda =\begin{pmatrix}\lambda_1& 0& \dots & 0 \\ 0 & \lambda_2 & \dots & 0 \\
\vdots & \vdots & \ddots & \vdots \\
0& 0& \dots & \lambda_n \end{pmatrix},
$$
where the $\lambda_i$ are the (real) eigenvalues of ${\tilde A}$ (and also of $A$). Inserting ${\tilde A}=Q\Lambda Q^T$ in the above Raleigh quotient formula yields
\begin{eqnarray*}
r(D\exponent^{\tau A})&=&\max_{x: \langle x,x \rangle =1} \langle x, D^{1/2}Q\exponent^{\tau \Lambda}Q^TD^{1/2}x \rangle \\
&=&\max_{x: \langle x,x \rangle =1} \langle Q^TD^{1/2}x, \exponent^{\tau \Lambda}Q^TD^{1/2}x \rangle \\
&=&\max_{y:\langle D^{-1/2}Qy,  D^{-1/2}Qy \rangle =1} \langle y,\exponent^{\tau \Lambda} y \rangle \\
&=&\max_{y: \langle y, Q^TD^{-1}Qy \rangle =1} \langle y,\exponent^{\tau \Lambda} y \rangle.
\end{eqnarray*}
For each fixed $y$ in the constraint set $\{y\, |\, \langle y, Q^TD^{-1}Qy \rangle =1\}$, the function 
$$
\tau \to \langle y, \exponent^{\tau \Lambda} y \rangle =\sum_{i=1}^N\exponent^{\lambda_i \tau}y_i^2 
$$ 
is a linear combination of exponential functions with non-negative weights, hence it is a convex function for $\tau \in [0,\infty)$. 
Since the function $\tau \to r(D\exponent^{\tau A})$, where $\tau \in [0,\infty)$, is a point-wise maximum of a collection of convex functions, it is also a convex function by Theorem $\ref{pointwise-max}$ in the Appendix.
\item Next we show that if in addition, $A$ is non-singular, then 
for any fixed $p>0$, the map $\tau \to r(D\exponent^{\tau A})$ is strongly convex with some parameter $m_p>0$, for $\tau \in [0,p]$.

We have shown that
$$
r(D\exponent^{\tau A})=\max_{y: \langle y, Q^TD^{-1}Qy \rangle =1} \langle y,\exponent^{\tau \Lambda} y \rangle,
$$
but here, the diagonal matrix $\Lambda$ is non-singular because its diagonal entries are the eigenvalues of $A$, which is non-singular by assumption.

Fix $p>0$. We will first show that there exists some $m_p>0$ such that for every 
$y$ in the constraint set $\{y\, |\, \langle y, Q^TD^{-1}Qy \rangle =1\}$ (this set is an ellipsoid because $D$ is a diagonal, positive-definite matrix), 
the map $f_y(\tau):=  \langle y,\exponent^{\tau \Lambda} y \rangle$, has the following property:
$$
f_y''(\tau) \geq m_p,\textrm{ for all } \tau \textrm{ in } [0,p].
$$

To see why, we calculate $f''_y$ for every $y$ in the constraint set:
$$
f''_y(\tau) = \sum_{i=1}^N \lambda_i^2\exponent^{\lambda_i \tau}y_i^2.
$$
Since all the $\lambda_i$ are non-zero, we have that 
for all $\tau$ in $[0,p]$
\begin{eqnarray*}
f''_y(\tau) &\geq& \sum_{i:\, \lambda_i <0} \lambda_i^2\exponent^{\lambda_i p}y_i^2 + \sum_{i:\, \lambda_i > 0} \lambda_i^2  y_i^2, \\
 &=&\sum_{i=1}^N b_i^2y_i^2,
\end{eqnarray*}
where for all $i=1,\dots, N$, we defined the positive numbers
$$
b^2_i=\begin{cases}
\lambda_i^2 \exponent^{\lambda_i p},\textrm{ if } \lambda_i<0 \\
\lambda_i^2,\textrm{ if } \lambda_i>0
\end{cases}.
$$
The quadratic form $y\to \sum_{i=1}^N b_i^2 y_i^2$ is clearly positive definite, hence it achieves its minimum $m_p$ over the compact constraint set $\{y\, |\, \langle y, Q^TD^{-1}Qy \rangle =1\}$. And as the latter set does not contain zero, this minimum $m_p$ is positive.

To summarize, we have shown that for each fixed $p>0$, there exists a positive $m_p$ such that $f''_y(\tau)\geq m_p$, for all $\tau$ in $[0,p]$, and for all $y$ in $\{y\, |\, \langle y, Q^TD^{-1}Qy \rangle =1\}$. Theorem $\ref{second-derivative}$ in the Appendix implies that for every $y$ in $\{y\, |\, \langle y, Q^TD^{-1}Qy \rangle =1\}$, the function 
$f_y(\tau)$ is strongly convex with parameter $m_p$, for $\tau$ in $(0,p)$ and thus also for $\tau$ in $[0,p]$. 
Theorem $\ref{pointwise-max}$ in the Appendix then implies that the map $\tau \to r(D\exponent^{\tau A})$ is strongly convex with parameter $m_p$, for $\tau$ in $[0,p]$.
\end{itemize}
\end{proof}

{\bf Remark}: We have seen that $\tau \to r(D\exponent^{\tau A})$ is convex for $\tau \in [0,\infty)$ whenever $D$ is a diagonal matrix with positive diagonal entries, and $A$ is diagonally symmetrizable. As mentioned earlier, the latter condition on $A$ implies that $A$ must have real eigenvalues. Here we give an example of a case where $A$ does not have real eigenvalues and which is such that $\tau \to r(D\exponent^{\tau A})$ is not convex for $\tau \in [0,\infty)$.

Let 
$$
D=\begin{pmatrix}
1&0 \\
0& d
\end{pmatrix}
\textrm{ and } 
A=\begin{pmatrix}
0& -1 \\
1& 0
\end{pmatrix},
$$
for some $d>0$. 
Then
$$
D\exponent^{\tau A}=
\begin{pmatrix}
\cos \tau & -\sin \tau \\
d\sin \tau & d\cos \tau
\end{pmatrix},
$$
and the eigenvalues of $D\exponent^{\tau A}$ are
$$
\lambda_{1,2}(D\exponent^{\tau A})=\begin{cases}
\frac{1+d}{2}\left(1\pm \sqrt{\cos^2 \tau - 4d/(d+1)^2} \right),\textrm{ if } \cos^2 \tau \geq \frac{4d}{(d+1)^2}\\
\frac{1+d}{2}\left(1\pm i\sqrt{4d/(d+1)^2-\cos^2 \tau} \right),\textrm{ if } \cos^2 \tau < \frac{4d}{(d+1)^2}
\end{cases},
$$
which implies that 
$$
r(D\exponent^{\tau A})=
\begin{cases}
\frac{1+d}{2}\left(1+ \sqrt{\cos^2 \tau - 4d/(d+1)^2} \right),\textrm{ if } \cos^2 \tau \geq \frac{4d}{(d+1)^2}\\
\frac{1+d}2\sqrt{\sin^2 \tau+4d/(1+d)^2},\textrm{ if } \cos^2 \tau < \frac{4d}{(d+1)^2}\\
\end{cases}
$$
{\bf Claim}: Assume that $d\neq 1$, and fix some $p\in (0,\pi/4)$ such that $p< \arccos (2\sqrt{d}/(d+1))$. Then $r(D\exponent^{\tau A})$ is strongly concave for $\tau$ in $[0,p]$ (a function $f$ is strongly concave if $-f$ is strongly convex). 

Indeed, note that for $\tau \in [0,p]$, $r(D\exponent^{\tau A})=f(g(\tau))$, where
$$
g(\tau)= \cos^2 \tau - \frac{4d}{(d+1)^2}\textrm{ and } f(x)=\frac{1+d}{2}\left(1+\sqrt x \right).
$$
Also note that $g$ is positive, decreasing and strongly concave on $[0,p]$ because $g(\tau)\geq g(p)>0$, $g'(\tau)=-\sin (2\tau)<0$, and 
$g''(\tau)=-2\cos (2\tau) \leq -2\cos(2p)<0$ for $\tau \in [0,p]$. Furthermore, for $x\in [g(p),g(0)]$, we have that $f'(x)=(1+d)/(4\sqrt{x})\geq (1+d)/(4g(p))>0$ and $f''(x)=-(1+d)/(8x^{3/2})\leq -(1+d)/(8(g(0))^{3/2})<0$.

Since for any twice continuously differentiable functions $f$ and $g$ holds that
$$
\frac{d^2}{d\tau^2} \left(f(g(\tau)) \right) = f''(g(\tau))(g'(\tau))^2+f'(g(\tau))g''(\tau),
$$
the above estimates imply that there is some $m>0$ such that
$$
\frac{d^2}{d\tau^2} \left(r(D\exponent^{\tau A}) \right) = \frac{d^2}{d\tau^2} \left(f(g(\tau)) \right)\leq -m,\textrm{ for all } \tau \in (0,p),
$$
and thus Theorem $\ref{second-derivative}$ in the Appendix implies that $r(D\exponent^{\tau A})$ is strongly concave on $[0,p]$, as claimed.

\section{Implications}\label{section_implications}

Our main convexity result provides useful information on the behavior of the map $\tau \to r(D\exponent^{\tau A})$ and sets us up to find an upper bound and optimal $\tau$ that ensures 0 is asymptotically stable for $(\ref{model})$. The convexity result along with an understanding of the boundary behavior of this map allows us to classify the behavior of the map $\tau \to r(D\exponent^{\tau A})$ based on straight forward properties of $A$ (Figure 1).

The following Lemma \ref{near-zero} and \ref{near-infinity} give a description of the boundary behavior.

\begin{lemma}\label{near-zero}
Assume that
\begin{enumerate}
\item $D$ is diagonal with positive diagonal entries and that 
there is a unique $k$ such that $D_{kk}> D_{ii}$ for all $i\neq k$, and
\item $A$ is a real $N\times N$ matrix.
\end{enumerate}
Then $\tau \to r(D\exponent^{\tau A})$ is continuously differentiable for all sufficiently small $\tau$. Furthermore,
$$
\frac{d}{d\tau} \left(  r(D\exponent^{\tau A} \right) |_{\tau =0}=D_{kk}A_{kk}.
$$
Consequently, if $A_{kk}<0$ (respectively $A_{kk}>0$), then 
$r(D\exponent^{\tau A})$ is strictly decreasing (respectively strictly increasing) for all sufficiently small $\tau$.
\end{lemma}
\begin{proof}
Note that the map $\tau \to B(\tau):=D\exponent^{\tau A}$ is continuously differentiable for all $\tau$ in $\reals$. Furthermore, 
$B(0)=D$ has a simple eigenvalue $D_{kk}=r(B(0))$ which is strictly larger than all other eigenvalues of $D$. This dominant eigenvalue
has corresponding left and right eigenvectors $u(0)=e_k$ and $v(0)=e_k$ (where $e_k$ is the $k$th standard basis vector in $\reals^N$) respectively, i.e. $B(0)u(0)=r(B(0))u(0)$ and $B^T(0)v(0)=r(B(0))v(0)$. It follows from the Implicit Function Theorem that there exists some $\epsilon>0$ such that for all $\tau$ with $|\tau| <\epsilon$, $r(B(\tau))$ is continuously differentiable, and that there exist differentiable vectors $u(\tau)$ and $v(\tau)$ such that
$$
B(\tau)u(\tau)=r(B(\tau))u(\tau)\textrm{ and } B^T(\tau)v(\tau)=r(B(\tau))v(\tau),
$$
which are normalized as follows
$$
\langle v(\tau),u(\tau) \rangle =1.
$$
Note that $r(B(\tau))=\langle v(\tau),B(\tau)u(\tau) \rangle$ for all 
$|\tau|<\epsilon$, and differentiating this identity yields 
\begin{eqnarray*}
\frac{d}{d\tau} \left( r(B(\tau)) \right) &=& r(\tau) \left(\langle \frac{d}{d \tau}\left(v(\tau)\right),u(\tau)\rangle + \langle v(\tau),\frac{d}{d\tau}\left(u(\tau) \right) \rangle \right)+\\
&&\langle v(\tau),\frac{d}{d \tau}\left(B(\tau)\right)u(\tau) \rangle \\
&=&\langle v(\tau),DA\exponent^{\tau A}u(\tau) \rangle,
\end{eqnarray*}
where the normalization identity was used to show that the first term above is zero. Evaluating at $\tau=0$ yields
$$
\frac{d}{d\tau} \left( r(B(\tau)) \right)|_{\tau =0} = \langle e_k,DAe_k \rangle =D_{kk}A_{kk},
$$
from which the conclusion follows.
\end{proof}

\begin{lemma}\label{near-infinity}
Assume that
\begin{enumerate}
\item $D$ is diagonal with positive diagonal entries, and
\item $A$ is a diagonally symmetrizable and let $\lambda_1$ be its largest eigenvalue.
\end{enumerate}
Then $\lim_{\tau \to +\infty} r(D\exponent^{\tau A})=\begin{cases}+\infty,\textrm{ if } \lambda_1>0\\
0,\textrm{ if }\lambda_1<0
\end{cases}$.
\end{lemma}
\begin{proof}
Recall from the proof of Theorem $\ref{main}$ that for all $\tau\geq 0$,
$$
r(D\exponent^{\tau A})=\max_{y: \langle y, Q^TD^{-1}Qy \rangle =1} \langle y,\exponent^{\tau \Lambda} y \rangle,
$$
where ${\tilde A}=T^{-1}AT$ for some real diagonal and invertible matrix $T$, and $Q$ is 
an orthogonal matrix such that $Q^T{\tilde A}Q=\Lambda$, for some 
diagonal matrix $\Lambda$ whose diagonal entries $\lambda_i$ are the eigenvalues of ${\tilde A}$ (and of $A$). Note that the ordering of the diagonal entries of $\Lambda$ is such that the entry in the top left corner of $\Lambda$ is $\lambda_1$, which is the largest eigenvalue of $A$.
\begin{itemize}
\item Assume that $\lambda_1>0$, and thus $A$ is unstable.
Let $\alpha$ be such that for $y_{\alpha}:=\alpha e_1$ holds that 
$\langle y_\alpha, Q^TD^{-1}Qy_\alpha \rangle =1$. Such a non-zero 
$\alpha$ exists because $Q^TD^{-1}Q$ is positive definite. 
Then it follows that for all $\tau\geq 0$,
$$
r(D\exponent^{\tau A})\geq \langle y_\alpha ,\exponent^{\tau \Lambda} y_\alpha \rangle =\alpha^2 \exponent^{\lambda_1 \tau},
$$
and since $\lambda_1>0$, the conclusion follows from taking the limit in the above inequality as $\tau \to +\infty$.
\item Assume that $\lambda_1<0$ and thus $A$ is stable. Then 
$\exponent^{\lambda_j \tau}\leq \exponent^{\lambda_1 \tau}$ for all $\tau\geq 0$ and for all $j$, because $\lambda_1$ is the largest eigenvalue of $A$. Then for all $\tau\geq 0$,
$$
r(D\exponent^{\tau A})=\max_{y: \langle y, Q^TD^{-1}Qy \rangle =1} \langle y,\exponent^{\tau \Lambda} y \rangle
\leq \exponent^{\lambda_1 \tau} \max_{y: \langle y, Q^TD^{-1}Qy \rangle =1}\langle y, y\rangle,
$$
and taking the limit as $\tau \to +\infty$ yields the desired result.
\end{itemize}
\end{proof}
The following is our main result concerning the effectiveness of periodic impulsive control of an unstable linear system 
${\dot x}=Ax$. It provides a sufficient condition on $A$ guaranteeing  
that this control methodology can stabilize the system, as long as the period between successive impulsive control interventions remains below a unique threshold. It also shows that there is a unique period $\tau$ for which $r(D\exponent^{\tau A})$ is minimized. This is important for  applications where the control objective is to stabilize the system as efficiently as possible.
\begin{stel}\label{main-stability}
Assume that
\begin{enumerate}
\item $D$ is diagonal with $D_{ii}\in (0,1]$ for all $i=1,\dots ,N$; furthermore, there is a unique $k$ such that $D_{kk}> D_{ii}$ for all $i\neq k$.
\item $A$ is a non-singular, diagonally symmetrizable and unstable matrix (i.e., the largest eigenvalue $\lambda_1$ of $A$ is positive), and $A_{kk}<0$.
\end{enumerate}
Then the function $\tau \to r(D\exponent^{\tau A})$, where $\tau \in [0,\infty)$, has the following properties:
\begin{enumerate}
\item There is a unique $\tau_s>0$ such that
$$
r(D\exponent^{\tau A})\begin{cases}
<1,\textrm{ if } \tau \in (0,\tau_s)\\
=1,\textrm{ if } \tau=\tau_s \\
>1,\textrm{ if } \tau > \tau_s
\end{cases}
$$
\item There is a unique $\tau_m\in (0,\tau_s)$ such that
$$
\inf_{\tau \geq 0} r(D\exponent^{\tau A})=r(D\exponent^{\tau_m A}).
$$
\end{enumerate}
\end{stel}
\begin{proof}
The map $\tau \to r(\tau):=r(D\exponent^{\tau A})$ is continuous for $\tau \geq 0$, and Lemma $\ref{near-zero}$ implies that it is strictly decreasing for all sufficiently small $\tau$ near zero. Thus $r(\tau)<r(0)=D_{kk}\leq 1$ for all sufficiently small $\tau>0$. 
Then by Lemma $\ref{near-infinity}$ and the Intermediate Value Theorem, there exists some $\tau_s>0$ such that $r(\tau_s)=1$ and $r(\tau)>1$ for all $\tau > \tau_s$. Furthermore, Theorem $\ref{main}$ implies that $r(\tau)$ is strongly convex with some parameter $m_{\tau_s}>0$ for $\tau\in [0,\tau_s]$. Hence $r(\tau)$ is strictly convex for $\tau\in [0,\tau_s]$, and this implies that 
$r(\tau)<1$ for all $\tau\in (0,\tau_s)$. This establishes item $(1)$.

Since $r(\tau)$ is continuous, it achieves its minimum on the compact set $[0,\tau_s]$, say for $\tau=\tau_m$, and clearly, $\tau_m \in (0,\tau_s)$ and $r(\tau_m)<1$. Theorem $\ref{unique-min}$ in the Appendix implies that $r(\tau)$ has a unique minimum at $\tau=\tau_m$ on $[0,\tau_s]$. But $r(\tau)>1$ for all $\tau>\tau_s$, and thus $r(\tau)$ also has a unique minimum at $\tau=\tau_m$ on $[0,\infty)$. This establishes item $(2)$, and concludes the proof.
\end{proof}

An example of the map $\tau \to r(\tau):=r(D\exponent^{\tau A})$ for given matrices $A$ and $D$ which illustrates the previous theorem is shown in the bottom left panel of Figure \ref{fig-A}. Next, we consider the other cases.

{\bf Remark} If all conditions of Theorem $\ref{main-stability}$ hold, except that $A_{kk}>0$ instead of assuming that $A_{kk}<0$, then the function $\tau \to r(\tau):=r(D\exponent^{\tau A})$ behaves differently. We claim that in this case, $r(\tau)$ is strictly increasing for $\tau\geq 0$, and 
$\lim_{\tau \to +\infty}r(\tau)=\infty$. The latter limit is immediate from Lemma $\ref{near-infinity}$. To see why $r(\tau)$ is strictly increasing, we argue by contradiction. If it were not, there would be some $0\leq \tau_1 <\tau_2$ such that $r(\tau_2)\leq r(\tau_1)$. Since by 
Lemma $\ref{near-zero}$, $r(\tau)$ is strictly increasing near zero, we may assume that $\tau_1>0$. But then $\tau_1$ belongs to the interval $(0,\tau_2)$ and since $r(\tau)$ is strictly convex on $[0,\tau_2]$ (by Theorem $\ref{main}$), there exists some $\lambda\in (0,1)$ such that 
$$
r(\tau_1)<\lambda r(0) + (1-\lambda) r(\tau _2).
$$
But since $r(\tau_2)\leq r(\tau_1)$, this implies that $r(\tau_1)< r(0)$. Now consider the interval $[0,\tau_1]$. Since $r(\tau)$ is strictly increasing near $\tau=0$, there exists some $\tau_3\in (0,\tau_1)$ such that $r(0)<r(\tau_3)$. 
Then strict convexity of $r(\tau)$ on $[0,\tau_1]$ (by Theorem $\ref{main}$) implies that
there is some $\mu\in (0,1)$ such that $r(\tau_3)<\mu r(0)+(1-\mu)r(\tau_1)<r(0)$ (because $r(\tau_1)<r(0)$), which contradicts that $r(0)<r(\tau_3)$.

Consequently, in this case, and assuming in addition that $r(0)=D_{kk}<1$, there will be a unique $\tau_s$ such that
$$
r(D\exponent^{\tau A})\begin{cases}
<1,\textrm{ if } \tau \in [0,\tau_s)\\
=1,\textrm{ if } \tau=\tau_s \\
>1,\textrm{ if } \tau > \tau_s
\end{cases},
$$
and $r(\tau)$ has its unique miminizer over $\tau \geq 0$ at $\tau_m=0$. An example illustrating this case is given in the top left of Figure \ref{fig-A}.

{\bf Remark} Theorem $\ref{main-stability}$ and the subsequent Remark consider the case where $A$ has a positive eigenvalue, and therefore is unstable. It is natural to ask what would happen if instead $A$ is assumed to be stable, although this scenario is perhaps 
less interesting from a control perspective. After all, why would one want to stabilize an already stable system with impulsive periodic control? We will show that in some sense, it is indeed best to not apply impulsive control in this case. We start with the following
\begin{lemma}\label{all-neg}
Assume that $A$ is a stable, diagonally symmetrizable matrix. Then $A_{ii}<0$ for all $i$.
\end{lemma}
\begin{proof}
As $A$ is diagonally symmetrizable, there exists a real, diagonal and invertible matrix $T$ such that
\begin{equation} \label{simil}
{\tilde A}=T^{-1}AT
\end{equation}
is a symmetric matrix. The eigenvalues of ${\tilde A}$ and $A$ coincide, and since $A$ is stable, this implies that the largest eigenvalue of ${\tilde A}$ is negative. Thus, ${\tilde A}$ is a negative semi-definite matrix. And since ${\tilde A}$ is symmetric, this implies that ${\tilde A}_{ii}<0$ for all $i$. Indeed, the largest eigenvalue $\lambda_1$ of ${\tilde A}$ is negative, and then Raleigh quotient
$$
\lambda_1=\max_{x:\langle x,x \rangle = 1} \langle x,{\tilde A}x \rangle,
$$
implies that for all $i$,
$$
\langle e_i, {\tilde A}e_i \rangle = {\tilde A}_{ii}\leq \lambda_1 <0,
$$
where $e_i$ denotes the $i$th standard basis vector of $\reals^N$. But note that since $T$ in $(\ref{simil})$ is a diagonal matrix, it follows that $A_{ii}={\tilde A}_{ii}$ for all $i$. Consequently, $A_{ii}<0$ for all $i$, which concludes the proof.
\end{proof}
The above Lemma is illustrated in the top right of Figure \ref{fig-A}. Equipped with this Lemma, we can now describe what happens when stabilizing a stable system with periodic impulsive control.
\begin{stel}\label{main-stability2}
Assume that
\begin{enumerate}
\item $D$ is diagonal with $D_{ii}\in (0,1]$ for all $i=1,\dots ,N$; furthermore, there is a unique $k$ such that $D_{kk}> D_{ii}$ for all $i\neq k$.
\item $A$ is a diagonally symmetrizable and stable matrix (i.e., the largest eigenvalue $\lambda_1$ of $A$ is negative).
\end{enumerate}
Then the function $\tau \to r(D\exponent^{\tau A})$ 
is strictly decreasing for $\tau \in [0,\infty)$, and $\lim_{\tau \to +\infty} r(D\exponent^{\tau A})=0$.
\end{stel}
\begin{proof}
For $\tau \geq 0$, we set $r(\tau):=r(D\exponent^{\tau A})$. By the variational description of $r(\tau)$, see for instance the second line of the proof of Lemma $\ref{near-infinity}$, we see that $r(\tau)>0$ for all $\tau\geq 0$. Since $A_{ii}<0$ for all $i$ by Lemma $\ref{all-neg}$, 
Lemma $\ref{near-zero}$ implies that $r(\tau)$ is strictly decreasing for all sufficiently small $\tau$, and Lemma $\ref{near-infinity}$ implies that $\lim_{\tau \to \infty} r(\tau)=0$. It remains to be shown that $r(\tau)$ is strictly decreasing for all $\tau \geq 0$. If this were not the case, then there would exist $0<\tau_1< \tau_2$ such that $0<r(\tau_1)\leq r(\tau_2)$. Since $\lim_{\tau \to \infty} r(\tau)=0$, there exists $\tau_3 > \tau_2$ such that $r(\tau_3)<r(\tau_1)$. Note that $\tau_2 \in (\tau_1, \tau_3)$, and then the strong convexity of $r(\tau)$ on $[\tau_1,\tau_3]$ (by Theorem $\ref{main}$) implies that there exists some $\lambda\in (0,1)$ such that
$$
r(\tau_2)<\lambda r(\tau_1)+ (1-\lambda)r(\tau_3)<r(\tau_1),
$$
which contradicts that $r(\tau_1)\leq r(\tau_2)$.
\end{proof}
Since $\tau \to r(D\exponent^{\tau A})$ is strictly decreasing on $[0,\infty)$ when the assumptions in Theorem $\ref{main-stability2}$ hold, this function has no minimizer. The limiting behavior of this function suggest that the minimizer ``occurs as $\tau$ approaches infinity", a scenario which corresponds to not applying any impulsive periodic control at all. An example is given in the bottom right of Figure \ref{fig-A}.

\begin{figure}
\begin{center}
    \includegraphics[trim={1cm 3cm 1cm 2cm},clip, scale=0.55]{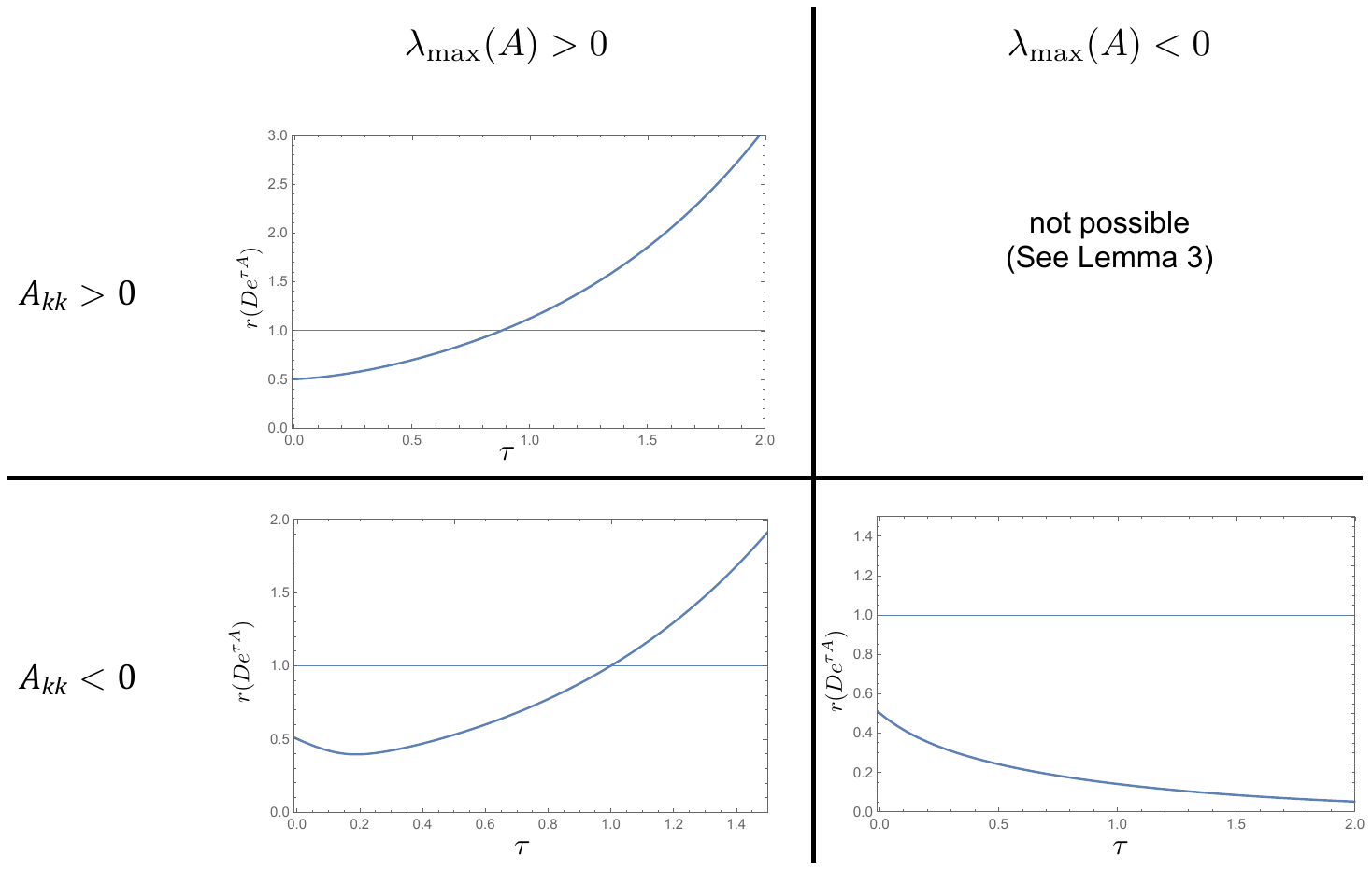}
\end{center}
\caption{Spectral radius as a function of $\tau$ in example matrices $A$ and $D$. All examples have a common matrix $D=\begin{bmatrix}
    0.5 & 0\\ 0 & 0.25
\end{bmatrix}$ but a different $A$. 
In the top left, bottom left, and bottom right, the matrices are $A=\begin{bmatrix}
    0.2 & 1\\ 1 & -0.2\end{bmatrix},
A=\begin{bmatrix}
   -2 & 1\\ 1 & 1\end{bmatrix},
A=\begin{bmatrix}
   -2 & 1\\ 1 & -2
\end{bmatrix}$, respectively. Here, $k=1$ since $D_{11}>D_{22}$, i.e., control for class 1 is weaker than for class 2.}\label{fig-A}
\end{figure}

To summarize, if we assume that the largest diagonal entry of $D$, $D_{kk}$, is less than $1$ (so control involves a multiplicative reduction in all variables), we can categorize the behavior into 3 distinct cases. 
First, if $\lambda_{\max}(A)<0$, the optimal strategy is to never control (bottom right in Figure \ref{fig-A}).  If $\lambda_{\max}(A)>0$ and the variable associated with weakest control (or largest diagonal entry $D_{kk}$ of $D$) is self-promoting (i.e., $A_{kk}>0$), then the optimal strategy is to treat as often as possible (top left in Figure \ref{fig-A}). Finally, if $\lambda_{\max}(A)>0$ and this variable is self-limiting (i.e., $A_{kk}<0$), then the optimal strategy is to treat at a precise frequency (bottom left in Figure \ref{fig-A}).



\section{Examples}\label{section_examples}
In this section, we provide an example of a real-world problem in biological control. which was previously discussed in the introduction as a motivating example.  For this, we apply our results to a given concrete model and discuss the implications to the ongoing control efforts. 

The model involves the biological control of a structured population, in which distinct population classes are impacted by the control intervention differently.  Before discussing the concrete example, we begin with a brief description of the results for an unstructured (1D) population model, for the purpose of juxtaposition.
\subsection{Unstructured populations}

In an unstructured population, we have the scalar linear impulse differential equation

\begin{align}
    \frac{dx}{dt} &= ax\\
x(n\tau^+) &= d x(n\tau^-)
\end{align}

The spectral radius is $d\exponent^{a\tau}$.  If $a>0$, then the map $\tau \rightarrow d \exponent^{a\tau}$ is strictly increasing, and hence, minimized when $\tau=0$. Conversely, if $a<0$, then the map is strictly decreasing. 

Interpreting this, if the population is exponentially growing in the absence of interventions, we should intervene as often as possible. Conversely, if the population is exponentially declining in the absence of interventions, we should not intervene. As we will see below, in case the populations are structured into classes and classes are impacted differently by the control intervention, there is no such simple dichotomy. 

\subsection{Spatially-structured: soil-transmitted helminths}
Soil-transmitted helminths are harmful parasitic worms that infect the gut of humans and several livestock species, causing various ailments such as malnutrition and developmental issues \cite{WHO}. They have been subjected to control efforts, including through preventative chemotherapy interventions, also known as massive drug administration (MDA).  This intervention involves the widespread distribution of oral anthelmintic drugs a few times a year to populations where the disease is thought to be prevalent. 

Here we focus on three human-infecting species: \emph{Ascaris lumbricoides} (herafter roundworm), \emph{Trichuris trichuria} (herafter whipworm), and \emph{Ancylostoma duodenale} (herafter hookworm). and compare the frequency of interventions needed for effective control.

\subsubsection{Model}
The life cycle of STH consists of two distinct spatial locations:
\begin{enumerate}
    \item Adult worms live and produce eggs within the host environment.
    \item Eggs get secreted into an external environment, where they develop into larvae.
\end{enumerate}

Following \cite{AndersonMaybook}, we model this through a set of linear differential equations.  Let $x_1(t)$ and $x_2(t)$ be the adult worms in the host environment and the larvae in the external environment, respectively.  Then, our uncontrolled model is


\begin{equation}
    \frac{d x}{dt} = \begin{bmatrix}
    -\mu_1 & \beta N \\
\lambda & -(\mu_2+ \beta N)
    \end{bmatrix}x
\end{equation}
    where $\beta$ is the per-host rate at which larvae are taken up (depends on their contact with the external environment), 
    $N$ is the population of hosts (``size of host environment"- assumed to be fixed), 
    $\lambda$ is the per-adult worm egg production rate,
    and $\mu_1$ and
    $\mu_2$ are the natural per-capita death rates in the host and external environment, respectively. 
These life cycle parameters differ between species.

We assume the only control intervention is the administration of anthelmintic drugs, and that this is given to a proportion $c$ of hosts in a given population periodically at intervals of length $\tau$.  The drug targets to kill adult worms within the host environment and has a different efficacy $\delta$; larvae in the external environment are not impacted.  Under these assumptions, our controlled model is of the form \ref{model}, with 
$$
D=\begin{bmatrix}
(1-c\delta) & 0\\
0 & 1
\end{bmatrix}
$$

\subsubsection{Parameterizations}
We use values as described in \cite{Truscott} and \cite{Coffeng} to parameterize the matrix $A$ for each species. Additionally, using drug efficacies from \cite{vercruysse2011}, (assuming egg count reduction is proportional to adult worm reduction) and assuming a $75\%$ coverage of school-aged children (target based on the WHO \cite{WHO_PC}) that compose of $50\%$ of the total population, we express the control intervention matrix $D$, for each species. 

For roundworms, we have

$$A_1 = \begin{bmatrix}
        -0.0028 & 1.3 \times 10^{-8} \\
    5000 & -0.016
\end{bmatrix}
\quad
D_1 = \begin{bmatrix}
    0.62875 & 0 \\
    0 & 1
\end{bmatrix}.
$$

For whipworms, we have

$$A_2 = \begin{bmatrix}
        -0.0028 & 2.089 \times 10^{-7} \\
    1000 & -0.05
\end{bmatrix}
\quad
D_2 = \begin{bmatrix}
    0.8125 & 0 \\
    0 & 1
\end{bmatrix}.
$$

For hookworms, we have
$$A_3 = \begin{bmatrix}
        -0.0014 & 1.18 \times 10^{-7} \\
    1500 & -0.082
\end{bmatrix}
\quad
D_3 = \begin{bmatrix}
    0.64375 & 0 \\
    0 & 1
\end{bmatrix}.
$$

The treatment drug is most effective against roundworms, which also has the lowest contact rate but the highest fecundity. The treatment drug is least effective against whipworms, but these also have the lowest fecundity.  Additionally, in all cases, the larvae class is the least controlled class, since it is not impacted by the intervention, i.e., $[D_i]_{22}>[D_i]_{11}$ for $i=1,2,3$. Hence, we set $k=2$.

\subsubsection{Implications}
Using each parameterization given above, we see that $\lambda_{\max}(A_i)>0$ for $i=1, 2, 3$, which predicts each parasite population will grow in the absence of any intervention.  Additionally, $A_i$ is diagonally symmetrizable, which follows from observing that $A_i$ is sign symmetric, that all $2\times 2$ matrices satisfy the cycle condition 
$(\ref{cycle})$ and, then invoking Proposition \ref{diag-symm-char}. Alternatively, one can use Corollary \ref{diag-symm-tri}. Finally, observe that $[A_i]_{kk}<0$ for $i=1,2,3$.  

Applying Theorem \ref{main-stability}, there is a unique period $\tau_m$ that minimizes the spectral radius map $\tau \rightarrow \rdeta$, for each species (Figure \ref{Fig_STH}).  The value $\tau_m$ is the optimal period to treat at.  For both whipworm and hookworm, the optimal period is very low, indicating that treatment is best administered every few days.  For roundworm, the optimal treatment period is around every 40 days.  

Additionally, applying Theorem \ref{main-stability}, there is a unique $\tau_s$ such that for all $\tau \in (0,\tau_s)$, we have $\rdeta <1$. The value $\tau_s$ is the upper bound to the amount of time in between drug administration needed in order to stabilize zero, i.e., effectively control the population.  We find that one needs to treat around once a year, more than twice a year or a little more than once every two years for roundworm, whipworm, and hookworm, respectively. The differences between the life history rates and the drug efficacies of these distinct species drive differences in the frequency of control interventions needed. Currently, the recommendations from the WHO for mitigating this disease in human populations are agnostic to which parasitic species are present \cite{WHO_PC}. 

\begin{figure}
\begin{center}
    \includegraphics[trim={0.1cm 5cm 0cm 5cm},clip, scale=0.6]{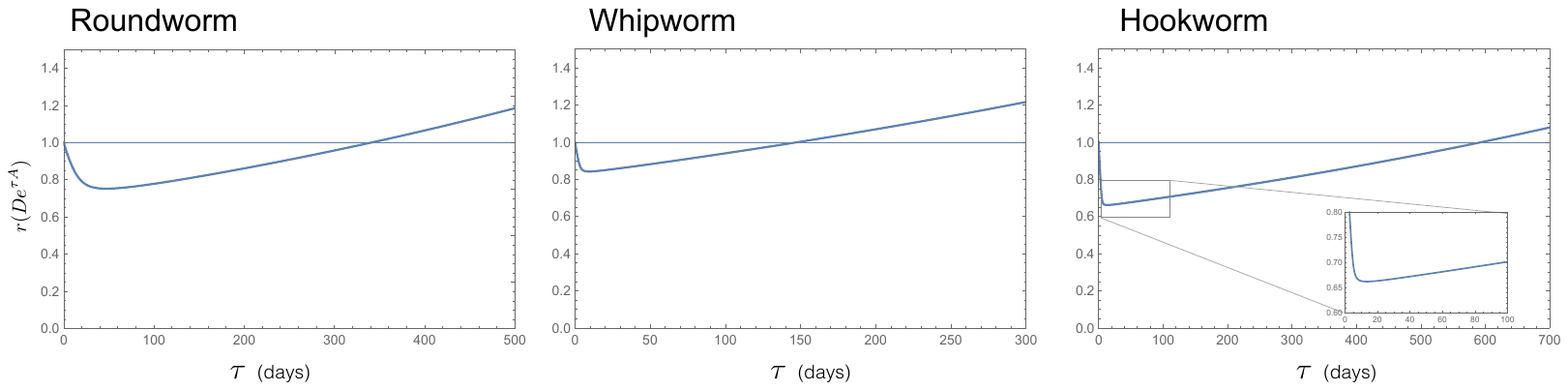}
\end{center}
\caption{Spectral radius of \deta as a function of $\tau$ (days) in the parameterized linear model for three species of soil-transmitted helminths. Parameter values are given in the main text.}\label{Fig_STH}
\end{figure}




\section{Discussion}\label{section_discussion}
Our main result establishes the convexity of the map $\tau \rightarrow \rdeta$.  This map arises from the consideration of the periodic impulse control of a system otherwise described by a set of linear differential equations. The main implications of this result are that:
\begin{itemize}
\item We find that there exists a unique range for the period between consecutive control interventions which guarantees that the system can be stabilized, and
\item We find a unique optimal period to achieve this.
\end{itemize}

Essentially, the matrix \deta comes from considering the discrete time map of the system at the moment of impulse, either right before or after the control intervention. Hence, it ignores the dynamics in between interventions, and rescales time to be relative to the period of the pulses. Because of this, one must be careful with the interpretation. In Theorem \ref{main-stability}, the optimal period (that minimizes the spectral radius $\rdeta$) is indicative of the most efficient period, i.e., how to get the ``closest" to zero per pulse.  It does not necessarily tell us about the system in real time and which control period will lead to zero the quickest. 

Our modeling framework assumes linearity of the system dynamics and, that the times between consecutive interventions are always equal. It is natural to aim to extend these results to consider nonlinear systems and more nuanced control strategies. By considering the linearization around an equilibrium point, the results may be applicable in some nonlinear systems.  Additionally, the approaches may be readily extended to other periodic control strategies.  For example, we could evaluate whether applying interventions multiple times in a short burst followed by a longer time in between bursts is more efficient than when they are all evenly spaced in time. Both considering nonlinear systems or these other strategies will require additional care and thought.

Finally, here we provide sufficient conditions, most crucially that $A$ is diagonally symmetrizable, for the convexity of the spectral radius of \deta. The following example demonstrates that this is not a necessary condition.  Consider an $N\times N$ triangular matrix with at least one nonzero off-diagonal entry.  This is clearly not diagonally symmetrizable, as it does not satisfy the sign-symmetric condition (see Proposition \ref{diag-symm-char}).  But the map $\tau \rightarrow \rdeta$ is convex.  To see this, observe that for a triangular matrix $A$ and a diagonal matrix $D$, the matrix $D\exponent^{\tau A}$ remains triangular and the eigenvalues are of the form $D_{ii}\exponent^{A_{ii}\tau}$. If $D$ is a non-negative matrix, then each eigenvalue is also non-negative, real, and convex with respect to $\tau$. Thus, the spectral radius  $r(D\exponent^{\tau})$ equals $\max_{i} \left( D_{ii}\exponent^{A_{ii}\tau}\right)$, which by Theorem \ref{pointwise-max} in the Appendix, 
is convex with respect to $\tau$. Future work will be to weaken or provide alternate conditions for this convexity.






\section*{Appendix}
In this Appendix we review some properties of convex functions. Most of the material presented here can be found in \cite{nesterov}, although for completeness, we include some proofs of results whose proof was omitted in \cite{nesterov}. We recall the definitions of three types of convexity. In what follows, $||x||$ denotes the Euclidean norm of the vector $x$ in $\reals^N$.
\begin{defi}
Let $C$ be a convex set in $\reals^N$. We say that a function $f:C\to \reals$ is
\begin{itemize}
\item convex if
$$
f(\lambda x_1 + (1-\lambda)x_2) \leq \lambda f(x_1)+(1-\lambda)f(x_2),\textrm{ for all } x_1,x_2 \textrm{ in } C\textrm{ and }\lambda \in [0,1].
$$
\item strictly convex if 
$$
f(\lambda x_1 + (1-\lambda)x_2) < \lambda f(x_1)+(1-\lambda)f(x_2),\textrm{ for all } x_1\neq x_2\textrm{ in }C\textrm{ and }\lambda \in (0,1).
$$
\item strongly convex \footnote{It is a standard exercise to show that this definition of strong convexity of $f$ with parameter $m>0$ is equivalent to $g(x):=f(x)-\frac{m}{2}||x||^2$ being convex.} with parameter $m>0$ if
$$
f(\lambda x_1 + (1-\lambda)x_2) \leq \lambda f(x_1)+(1-\lambda)f(x_2) -\frac{m}{2}\lambda(1-\lambda)||x_1-x_2||^2,
$$
for all  $x_1$ and $x_2$ in $C$, and all $\lambda$ in $[0,1]$.
\end{itemize}
\end{defi}

For continuously differentiable functions, the next two results provide characterizations of convexity and strong convexity.
\begin{stel} (Theorem 2.1.2 in \cite{nesterov}) \label{tangent-plane}
Let $C$ be an open convex set in $\reals^N$, and assume that $f:C\to \reals$ is continuously differentiable (i.e., the gradient $\nabla f(x)$ is continuous for all $x$ in $C$). Then $f$ is convex  if and only if
\begin{equation}\label{plane}
f(x_2)\geq f(x_1)+\langle \nabla f(x_1),x_2-x_1 \rangle, \textrm{ for all }x_1,x_2 \textrm{ in }C.
\end{equation}
\end{stel}



\begin{stel} (Theorem 2.1.9 in \cite{nesterov})   \label{characterizations}
Let $C$ be an open convex set in $\reals^N$, and assume that $f:C\to \reals$ is continuously differentiable.
Then the following statements are equivalent:
\begin{itemize}
\item $f$ is strongly convex with parameter $m>0$. 
\item \begin{equation}\label{second}
f(x_2)\geq f(x_1)+ \langle \nabla f(x_1), x_2-x_1 \rangle +\frac{m}{2}||x_2-x_1||^2,\textrm{ for all } x_1\textrm{ and } x_2 \textrm{ in }C.
\end{equation}
\item \begin{equation}\label{first}
\langle \nabla f(x_1) - \nabla f(x_2), x_1-x_2 \rangle \geq m ||x_1-x_2||^2,\textrm{ for all } x_1\textrm{ and } x_2 \textrm{ in }C.
\end{equation}
\end{itemize}
\end{stel}
\begin{proof}
\begin{itemize}
\item Let's first show  the equivalence of strong convexity of $f$ with parameter $m>0$ and condition $(\ref{second})$.
If $f$ is strongly convex with parameter $m>0$, then for all $x_1$ and $x_2$ in $C$, and for all $\lambda$ in $[0,1)$ (note that we are excluding $\lambda=1$),
\begin{eqnarray*}
 f(x_2)&\geq&  \frac{f(\lambda x_1 + (1-\lambda)x_2)}{1-\lambda} - \frac{\lambda}{1-\lambda}f(x_1) + \frac{m}{2} \lambda ||x_1-x_2||^2 \\
 &=&f(x_1)+\frac{f(\lambda x_1 +(1-\lambda)x_2) - f(x_1)}{1-\lambda}    + \frac{m}{2} \lambda ||x_1-x_2||^2\\
 &=&f(x_1)+\frac{f(x_1 +(1-\lambda)(x_2-x_1)) - f(x_1)}{1-\lambda} + \frac{m}{2} \lambda ||x_1-x_2||^2
\end{eqnarray*}
Taking the limit as $\lambda\to 1-$ in the above inequality, the limit of the last quotient  is  equal to  $\langle \nabla f(x_1),x_2-x_1 \rangle$. 
This shows that $(\ref{second})$ holds.

For the converse, note that for all $x_1$ and $x_2$ in $C$ and all $\lambda$ in $[0,1]$,
\begin{eqnarray*}
f(x_2)&\geq& f(\lambda x_1 + (1-\lambda)x_2) + \langle \nabla f (\lambda x_1 + (1-\lambda)x_2), \lambda (x_2-x_1) \rangle \\
&+& \frac{m}{2} || \lambda (x_2-x_1)||^2\\
f(x_1)&\geq& f(\lambda x_1 + (1-\lambda)x_2) + \langle \nabla f (\lambda x_1 + (1-\lambda)x_2), (1-\lambda) (x_1-x_2) \rangle  \\
&+& \frac{m}{2} || (1-\lambda ) (x_2-x_1)||^2
\end{eqnarray*}
Multiplying the first inequality by $(1-\lambda)$, the second by $\lambda$, and adding the resulting inequalities shows that $f$ is strongly convex with parameter $m>0$.

\item Let's show next that $(\ref{second})$ and $(\ref{first})$ are equivalent. If $(\ref{second})$ holds then for all $x_1$ and $x_2$ in $C$,
\begin{eqnarray*}
f(x_2)&\geq& f(x_1)+ \langle \nabla f(x_1), x_2-x_1 \rangle +\frac{m}{2}||x_2-x_1||^2,\textrm{ and }\\
f(x_1)&\geq& f(x_2)+ \langle \nabla f(x_2), x_1-x_2 \rangle +\frac{m}{2}||x_2-x_1||^2,
\end{eqnarray*}
Adding both inequalities yields $(\ref{first})$. 

For the converse, assume that $(\ref{first})$ holds. Then for all $x_1$ and $x_2$ in $C$,
\begin{eqnarray*}
f(x_2)-f(x_1) &=& \int_0^1 \frac{d}{d\tau} \left( f(x_1+\tau(x_2-x_1))\right)d \tau \\
&=&\int_0^1 \langle \nabla f (x_1+\tau(x_2-x_1)), x_2-x_1 \rangle d \tau \\
&=&\langle \nabla f (x_1),x_2-x_1 \rangle + \int_0^1\frac{1}{\tau} \langle \nabla f (x_1+\tau(x_2-x_1)) - \nabla f(x_1), \tau(x_2-x_1) \rangle d \tau \\
&\geq& \langle \nabla f (x_1),x_2-x_1 \rangle + \int_0^1\frac{1}{\tau} m||\tau(x_2-x_1)||^2d \tau \\
&=&\langle \nabla f (x_1),x_2-x_1 \rangle + \int_0^1 \tau m ||\tau(x_2-x_1)||^2d \tau \\
&=&\langle \nabla f (x_1),x_2-x_1 \rangle + \frac{m}{2}||x_2-x_1||^2,
\end{eqnarray*}
and thus $(\ref{second})$ holds.
\end{itemize}
\end{proof}

Next we review characterizations of convexity and strong convexity for functions which are twice continuously differentiable. 
For such functions $f:C\to \reals$, we denote the Hessian of $f$ at any $x$ in $C$ by $H_f(x)$, i.e. $[H_f(x)]_{ij}=\partial^2 f/\partial x_i \partial x_j (x)$, for all $i,j$ in $\{1,\dots, N\}$ and all $x$ in $C$. For a symmetric $N\times N$ matrix $B$, the notation $B\succeq 0$ simply means that $B$ is a positive semi-definite matrix (or equivalently, that all the eigenvalues of $B$ are non-negative real numbers).
\begin{stel}  (Theorems 2.1.4 and 2.1.11 in \cite{nesterov})) \label{second-derivative}
Assume that $C$ is an open convex set in $\reals^N$, and that $f:C\to \reals$ is twice continuously differentiable. Then
\begin{itemize}
\item $f$ is convex if and only if $H_f(x)\succeq 0$ for all $x$ in $C$.

\item $f$ is strongly convex with parameter $m>0$ if and only if $H_f(x) - mI\succeq 0$, for all $x$ in $C$.
\end{itemize}
\end{stel}
\begin{proof}
We only provide the proof of the last statement. 

Assuming that $f$ is twice continuously differentiable and strongly convex with parameter $m>0$, it follows from $(\ref{first})$ in Theorem $\ref{characterizations}$ that 
for all sufficiently small $\tau >0$ and for every $s$ in $\reals^N$, 
\begin{eqnarray*}
0&\leq& \frac{1}{\tau} \left( \langle \nabla f(x+\tau s) - \nabla f(x), \tau s \rangle -m\langle \tau s, \tau s \rangle\right)\\
&=& \langle \nabla f (x+\tau s) -\nabla f(x),s \rangle -m\tau \langle s,s \rangle \\
&=&\int_0^\tau \frac{d}{d\lambda} \left( \langle \nabla f(x+\lambda s),s \rangle \right) d\lambda - \int_0^\tau \langle ms,s \rangle d\lambda \\
&=&\int_0^\tau \langle H_f(x+\lambda s)s,s \rangle -\langle ms,s \rangle d\lambda \\
&=&\int_0^\tau \langle \left(H_f(x+\lambda s) -mI\right)s,s \rangle d\lambda
\end{eqnarray*}
As the integrand in the last integral is continuous in $\lambda$, it follows from taking the limit as $\tau\to 0+$ that
$$
0\leq \langle \left(H_f(x) -mI\right)s,s \rangle,\textrm{ for all }s \textrm{ in }\reals^N.
$$
In other words, $H_f(x)-mI \succeq 0$, for all $x$ in $C$.

For the converse, assume that $f$ is twice continuously differentiable and that there is some $m>0$ such that $H_f(x)-mI \succeq 0$ for all $x$ in $C$. 
Then for all $x$ and $y$ in $C$,
\begin{eqnarray*}
f(y)&=&f(x) + \int_0^1  \frac{d}{d\tau}  \left( f (x+\tau (y-x)) \right) d \tau \\
&=&f(x) +  \int_0^1 \langle \nabla f(x+ \tau (y-x)), y-x \rangle d \tau\\
&=&f(x)+ \langle \nabla f (x), y-x \rangle +  \int_0^1 \int_0^\tau \frac{d}{d \lambda} \left( \langle \nabla f(x+ \lambda (y-x)), y-x \rangle \right) d \lambda d \tau\\
&=&f(x)+ \langle \nabla f (x), y-x \rangle +  \int_0^1 \int_0^\tau \langle H_f(x+\lambda (y-x))(y-x), y-x \rangle d \lambda d \tau \\
&\geq &f(x)+ \langle \nabla f (x), y-x \rangle + \int_0^1 \int_0^\tau \langle m (y-x),y-x \rangle   \rangle d \lambda d \tau \\
&=&f(x)+ \langle \nabla f (x), y-x \rangle + \frac{m}{2} ||y-x||^2
\end{eqnarray*}
Thus, we have shown that $(\ref{second})$ holds, and therefore Theorem $\ref{characterizations}$ implies that $f$ is strongly convex with parameter $m>0$.
\end{proof}

In particular, if $C$ is an open and convex subset of $\reals$ and if $f: C\to \reals$ is a twice continuously differentiable function, then $f$ is strongly convex with 
parameter $m>0$ if and only if $f''(x)\geq m$ for all $x$ in $C$. For example, this shows that if $f: (-1,1)\to \reals$ is defined as $f(x)=x^4$, then $f$ is not strongly convex with respect to any positive parameter $m>0$, because $f''(0)=0$.

Next we show that the point-wise supremum of a collection of convex  functions is 
also convex, and furthermore, that the point-wise supremum of a collection of strongly convex functions all having the same parameter, is also strongly convex with the  same parameter.
\begin{stel}\label{pointwise-max}
Let $C$ be a convex set, and assume that $f_\alpha:C\to \reals$ is convex, for all $\alpha$ in some (possibly infinite) index set $I$. Then 
$$
g(x):=\sup_{\alpha \in I} f_\alpha(x)
$$
is convex.

Furthermore, if there is some $m>0$ such that for all $\alpha$ in $I$, the function $f_\alpha$ is strongly convex with parameter $m$, then $g$ is also strongly convex with parameter $m$.
\end{stel}
\begin{proof}
Assume first that all $f_\alpha$ are convex, for all $\alpha$ in $I$. Then for all $\alpha$ in $I$, all $x_1$ and $x_2$ in $C$ and $\lambda$ in $[0,1]$, holds that 
\begin{eqnarray*}
f_\alpha(\lambda x_1 +(1-\lambda)x_2)&\leq & \sup_{\alpha \in I} \left(\lambda f_\alpha(x_1)+(1-\lambda)f_\alpha (x_2) \right) \\
&\leq&\lambda \sup_{\alpha \in I} f_\alpha (x_1) + (1-\lambda)\sup_{\alpha \in I}f_\alpha (x_2)
\end{eqnarray*}
Taking the supremum over all $\alpha$ in $I$ yields that $g$ is convex.

Next assume that there is some $m>0$ such that
every $f_\alpha$ is strongly convex with parameter $m$, for all $\alpha$ in $I$. Then for all $\alpha$ in $I$, all $x_1$ and $x_2$ in $C$ and $\lambda$ in $[0,1]$, holds that 
\begin{eqnarray*}
f_\alpha(\lambda x_1 +(1-\lambda)x_2)&\leq & \sup_{\alpha \in I} \left(\lambda f_\alpha(x_1)+(1-\lambda)f_\alpha (x_2) \right) -\frac{1}{2}m\lambda(1-\lambda)||x_1-x_2||^2\\
&\leq&\lambda \sup_{\alpha \in I} f_\alpha (x_1) + (1-\lambda)\sup_{\alpha \in I} f_\alpha (x_2) -\frac{1}{2}m\lambda(1-\lambda)||x_1-x_2||^2.
\end{eqnarray*}
Taking the supremum over all $\alpha$ in $I$ yields that $g$ is strongly convex with parameter $m$.

\end{proof}

It is natural to ask if the point-wise supremum of a collection of strictly convex functions is strictly convex. While it is easy to see that this is true if the collection is finite, it is not necessarily true if the collection is infinite. Example: Let $n\in \naturals$ and suppose that $f_n:[-1,1] \to \reals$ is defined as $f_n(x)=(x^2-1)/n$. Clearly, each $f_n$ is strictly convex, but $g(x)=\sup_{n} f_n(x)=0$ for all $x$ in $[-1,1]$, and this function is convex but not strictly convex.

Finally, we 
recall the importance of strict convexity on the uniqueness of minimizers.

\begin{stel}\label{unique-min}
Let $C$ be a convex set in $\reals^N$, 
and assume that $f:C\to \reals$ is strictly convex. 
Then $f$ has at most one global minimizer in $C$.
\end{stel}
\begin{proof}
Assume that $x_1$ and $x_2$ are two distinct global minimizers of $f$ in $C$. Then $f(x_1)=f(x_2)$. However, since $f$ is strictly convex, it follows that $f(\lambda x_1 +(1-\lambda)x_2)<\lambda f(x_1)+(1-\lambda) f(x_2)=f(x_1)$, for all $\lambda$ in $(0,1)$. This contradicts that $x_1$ is a global minimizer of $f$. 
\end{proof}


\newpage
\begin{thebibliography}{199}

\bibitem{AndersonMaybook} R.M Anderson and R.M. May, Infectious diseases of humans: dynamics and control, Oxford Science Publications, 1991.

\bibitem{chicone} C. Chicone, Ordinary differential equations with applications, 2nd Edition, Springer, 2006.

\bibitem{Coffeng} L.E. Coffeng, J.E. Truscott, S.H. Farrell, H.C. Turner, R. Sarkar, G. Kang, S.J. de Vlas, and  R.M. Anderson, 2017. Comparison and validation of two mathematical models for the impact of mass drug administration on Ascaris lumbricoides and hookworm infection. Epidemics, 18, p.38-47.

\bibitem{cohen1978}J.E. Cohen, Derivatives of the spectral radius as a function of non-negative matrix elements. Math. Proc. Camb. Phil. Soc. 83, p.183-190, 1978.

\bibitem{cohen1979}J.E. Cohen, Random evolutions and the spectral radius of a non-negative matrix. Math. Proc. Camb. Phil. Soc. 86, p.345-350, 1979.

\bibitem{cohen1981} J.E. Cohen, Convexity of the dominant eigenvalue of an 
essentially nonnegative matrix, Proceedings of the AMS 81 (4), p.657-658, 1981.

\bibitem{donsker} M.D. Donsker, and R.S. Varadhan, On a variational formula for the principal eigenvalue for operators with maximum principle. Proceedings of Nat. Acad. Sci USA 72, p. 780-783, 1975. 

\bibitem{friedland81}
S. Friedland, Convex spectral functions, 
Linear and Multilinear Algebra 9, p. 299-316, 1981.

\bibitem{Kao} R. Kao,  D. Leathwick, M. Roberts, I. and Sutherland, Nematode parasites of sheep: A survey of epidemiological parameters and their application in a simple model. Parasitology, 121(1), p. 85-103, 2000.

\bibitem{mckee} J. McKee, and C. Smyth, 
Symmetrizable integer matrices having all their eigenvalues in the 
interval $[-2,2]$, Algebraic Combinatorics 3 (3), p.775-789, 2020.

\bibitem{nesterov} Y. Nesterov, Introductory Lectures on Convex Optimization, Springer, 2004.

\bibitem{trotter} H.F. Trotter, On the product of semigroup operators, Proceedings of AMS. 19, p. 545-551, 1959.

\bibitem{Truscott} J.E. Truscott,  H.C. Turner, S.H. Farrell,  and R.M. Anderson. Soil-transmitted helminths: mathematical models of transmission, the impact of mass drug administration and transmission elimination criteria. Advances in parasitology, 94, p.133-198, 2016.


\bibitem{vercruysse2011} Vercruysse, J., Behnke, J.M., Albonico, M., Ame, S.M., Angebault, C., Bethony, J.M., Engels, D., Guillard, B., Hoa, N.T.V., Kang, G. and Kattula, D., 2011. Assessment of the anthelmintic efficacy of albendazole in school children in seven countries where soil-transmitted helminths are endemic. PLoS neglected tropical diseases, 5(3), p.e948.

\bibitem{WHO_PC} World Health Organization. Guideline: preventative chemotherapy to control soil-transmitted helminth infections in at-risk population groups. Geneva: World Health Organization, 2017, Licence: CC BY-NC-SA 3.0 IGO.

\bibitem{WHO} World Health Organization. Soil-transmitted helminth infections. Retrieved October 5, 2023, from https://www.who.int/news-room/fact-sheets/detail/soil-transmitted-helminth-infections

\end{thebibliography}
\end{document}